\documentclass[12pt, A4,reqno]{amsart}
\usepackage{amsmath}
\usepackage{amssymb}
\usepackage{amsfonts,color}
\numberwithin{equation}{section}
\usepackage{a4wide}
\usepackage{graphicx}
\usepackage{comment}
\usepackage{accents}

\newcommand{\dbtilde}[1]{\accentset{\approx}{#1}}

\newcommand{\CC}{\mathbb{C}}
\newcommand{\DD}{\mathbb{D}}
\newcommand{\NN}{\mathbb{N}}
\newcommand{\RR}{\mathbb{R}}

\newcommand{\ZZ}{\mathbb{Z}}

\renewcommand{\tilde}{\widetilde}

\renewcommand{\Im}{\operatorname{Im}}

\newcommand{\bea}{\begin{eqnarray}}
\newcommand{\eea}{\end{eqnarray}}

\newcommand{\beqa}{\begin{eqnarray*}}
\newcommand{\eeqa}{\end{eqnarray*}}
\newcommand{\D}{\displaystyle}

\newcommand{\lra}{\longrightarrow}
\newcommand{\eps}{\varepsilon}

\DeclareMathOperator{\Hol}{Hol}

\DeclareMathSymbol{\subsetneqq}{\mathbin}{AMSb}{36}

\DeclareMathOperator{\Pol}{Pol}
\DeclareMathOperator{\Area}{Area}

\newtheorem{thm}{Theorem}
\newtheorem{theorem}{Theorem}
\newtheorem{lem}[theorem]{{\bf Lemma}}
\newtheorem{coro}{{\bf Corollary}}
\newtheorem{prop}{{Proposition}}
\newtheorem{remark}{{Remark}}
\newtheorem{example}{{Example}}

\title{Dominating sets in Bergman spaces and sampling constants}

\author{A. Hartmann, D. Kamissoko, S. Konate \& M.-A. Orsoni}

\subjclass[2010]{30H20}
\keywords{Dominating sets, reverse Carleson measure, Bergman space, sampling}
\thanks{The research of the first and the last author is partially supported by the project ANR-18-CE40-0035 and by the Joint French-Russian Research Project PRC CNRS/RFBR 2017--2019.  The research of the third author is partially supported by Banque Mondiale via the Projet d'Appui
au D\'eveloppement de l'Enseignement Sup\'erieur du Mali.}

\address{Univ. Bordeaux, CNRS, Bordeaux INP, IMB, UMR 5251,  F-33400, Talence, France}  
\email{Andreas.Hartmann@math.u-bordeaux1.fr}
\email{Marcu-Antone.Orsoni@math.u-bordeaux1.fr}

\address{University of Segou}
\email{gnatiosia@gmail.com}
\email{dan.kamiss@yahoo.fr}

\begin{document}
\begin{abstract}
We discuss sampling constants for dominating sets in Bergman spaces. 
Our method is based on a Remez-type
inequality by Andrievskii and Ruscheweyh. 
We also comment on extensions of the method to other spaces such as Fock and Paley-Wiener
spaces.
\end{abstract}
\maketitle

\section{Introduction}

Let $A^{p,\alpha}(\DD)$, $1\le p<\infty$, $\alpha>-1$ 
be the  weighted Bergman space on the unit disk $\DD$ in 
the complex plane defined by
\bea\label{Samp1}
 A^{p,\alpha}(\DD)=\{f\in \Hol(\DD):\|f\|_p^p=(\alpha+1)\int_{\DD}|f(z)|^p(1-|z|^2)^{\alpha}
dA(z)<+\infty\}.
\eea
Here $dA$ denotes normalized Lebesgue area measure on $\DD$. We will also use the
notation $dA_{\alpha}(z)=(\alpha+1)(1-|z|^2)^{\alpha}dA(z)$ (see \cite{HKZ} for more
information on Bergman spaces). The unweighted Bergman space is denoted by $A^p=A^{p,0}$.
In this paper we are interested 
in  measurable 
sets $E\subset \DD$ for which we have
\bea\label{dom}
 \int_E|f(z)|^pdA_{\alpha}(z)\ge C^p \int_{\DD}|f(z)|^p dA_{\alpha}(z).
\eea
We will occasionally write $L^{p,\alpha}(F)$ for the Lebesque space on a measurable set $F
\subset \DD$ with respect to $dA_{\alpha}$.
A set satisfying \eqref{dom} will be called {\it dominating} for $A^{p,\alpha}$. We mention that 
dominating sets  are closely related to so-called reverse Carleson measures for which we refer 
to the survey \cite{FHR} for more information.
The question of dominating sets in the Bergman space has been studied first by Luecking 
\cite{Lu1}-\cite{Lu3} who completely 
characterized these dominating sets in $A^p$ in terms of {\it relative density}, which 
morally speaking means that every pseudohyperbolic disk of a certain minimal fixed radius (depending
on $E$) meets the set $E$ with uniformly positive density (precise definitions will be given below).
We mention the closely related 
question of sampling sequences which have been characterized by Seip (see the monograph 
\cite{S} for a general reference on sampling and interpolation). 

A question which arises quite naturally is whether it is possible to estimate the constant $C$ 
appearing in \eqref{Samp1} in terms of the underlying density. A prominent example where 
this question has been studied is the
Paley-Wiener space. Dominating sets in this space go back to work by Logvinenko-Sereda \cite{LS}
and Panejah \cite{Pa1}-\cite{Pa2}. We mention also some recent work on so-called
Paley-Wiener measures (see e.g. \cite[Theorem 15]{Po}).
Having precise estimates on the 
sampling constants is important in applications when one has to decide on the trade-off between
the cost of the sampling and the accuracy of the estimates. In the early 2000's 
Kovrijkine \cite{Ko} considered the Paley-Wiener space and gave a precise and optimal 
estimate on $C$ in terms of the underlying density. His method is quite clever
and based on the Bernstein inequality which does not hold in Bergman or Fock spaces
but for instance in model spaces where the proof was adapted in \cite{HJK}. It was brought to our attention that very recently, 
Jaming and Speckbacher \cite{JS} considered
the case of polyanalytic Fock spaces using a strong result by Brudnyi --- which holds for
plurisubharmonic functions --- and a clever
trick allowing to switch from polyanalytic functions to analytic ones in 2 variables. Since the
Fock space is a special case of the polyanalytic Fock space, they in particular
get an estimate of the sampling constant in Fock spaces.
Our method 
is more elementary and uses  holomorphy. While our result is presented for the Bergman space,
for which the geometry is different, the method applies quite directly to consider the case of
the (analytic) Fock space. 
It should also be noted that our intermediate step (Proposition \ref{prop1}) allows to turn
around the Bernstein inequality in a large set of situations, including of course Fock spaces
but also the Paley-Wiener space.
We will comment on these observations in the last section. 
\\

We need to introduce some notation.
Let 
\[
 \rho(z,w)=\left|\frac{z-w}{1-\overline{z}w}\right|
\]
be the pseudohyperbolic distance between two points $z,w\in\DD$. We consider the
associated pseudohyperbolic balls: $D_{phb}(z,r)=\{w\in \DD:\rho(z,w)<r\}$, where 
$0<r<1$. A measurable set $E\subset\DD$ is called $(\gamma,r_0)$-dense for some
$\gamma>0$ and $0<r_0<1$, if for every $z\in\DD$
\[
 \frac{|E\cap D_{phb}(z,r_0)|}{|D_{phb}(z,r_0)|}\ge \gamma.
\]
Here $|F|$ denotes planar Lebesgue measure of a measurable set $F$.
We will just say that the set is relatively dense if there is some $\gamma>0$ and some
$0<r<1$ such that the set is $(\gamma,r)$-dense.\\

Luecking's result on dominating sets can be stated as follows (\cite{Lu1}, and in particular condition
(2') in \cite[p.4]{Lu1}). 

\begin{thm}[Luecking]\label{L-early}
A (Lebesgue) measurable set $E$ is dominating in $A^p$ ($p>0$) if and only if 
it is relatively dense.
\end{thm}

Looking closely at the proof proposed by Luecking --- who was not interested in the magnitude
of the sampling constant --- it turns out that his sampling constant behaves like
$C^p\ge \gamma e^{-c_p/\gamma}$. A more general result concerning sampling measures in a large class of spaces of analytic funtions has been discussed again by Luecking in \cite{Lu3}, and in particular in Bergman spaces \cite[Theorem 1]{Lu3}. Even if the density appears explicitely in his proof, a compactness argument used in \cite[Lemma 4]{Lu3} introduces an implicit dependence on the density.

Our main result is the following

\begin{thm}\label{thm1}
Let $1\le p<+\infty$. There exists $L$ such that for every measurable set 
$E\subset \DD$ 
which is $(\gamma,r)$-dense, we have
\bea\label{estim1}
  \|f\|_{L^{p,\alpha}(E)}\ge \left(\frac{\gamma}{c}\right)^L \|f\|_{A^{p,\alpha}}
\eea
for every $f\in A^p$. 
\end{thm}

The constants $c$ and $L$ depend on $r$.
For $L$ we can choose
\[
 L=c_1\frac{1+\alpha}{p}\frac{1}{(1-r)^4}\ln\frac{1}{1-r},
\]
where $c_1$ is some universal constant.
\\

It should be noted that there is a competing relation between $\gamma$ and $r$. The
density $\gamma$ can be very small for a given $r$ (because $E$ has holes which  have
pseudohyperbolic radius close to $r$), but it can become rather big when we choose a bigger
radius (e.g.\ pseudohyperbolic doubling of $r$: $2r/(1+r^2)$). 
In a sense one needs to optimize $L\ln \gamma$, and $L$ depends on $r$. 
\\

Here is another observation.
Though this might be obvious, it should be observed that there is no reason 
a priori why a holomorphic function for which the integral $\int_E|f|^pdA_{\alpha}$ 
is bounded  for a  relative dense set $E$ should be in $A^{p,\alpha}$. 
Outside the class $A^{p,\alpha}$ relative density is 
in general not necessary for domination (see also a remark in  \cite[p.11]{Lu1}).
\\

The discussion of the necessary condition of relative density in \cite[p.5]{Lu1}, involving 
testing on reproducing kernels, shows that
given a sampling constant $C$ in \eqref{dom}, then $\gamma\gtrsim C^{p}$, i.e.
$C\lesssim \gamma^{1/p}$.
\\

There are three main ingredients in the proof of Theorem \ref{thm1}. 
The first one is a Remez-type inequality, which allows to get
the local estimate depending on the density. These have been studied in a large set of situations.
We refer to \cite{BE} for a general source on polynomials and Remez-type inequalities, and to 
\cite{E} for a survey. In our situation, we need a Remez-type inequality for planar domains, which
can be found in \cite{AR}. The second ingredient is to decompose the integration domain into
good and bad parts. Good meaning here that the Remez-type inequality applies. In Kovrijkine's
work, as well as in many of the succeeding work using his method, the separation in good and bad
parts was achieved {\it via} a Bernstein-type inequality. Such an inequality holds for instance in 
Paley-Wiener spaces and more generally in so-called model spaces. However, it is no longer true
in Fock or Bergman spaces. So a different approach has to be found to get the good intervals 
(rather disks in Fock or Bergman spaces).
The way of turning around Bernstein's inequality given in this paper is one of the new features and applies to many other situations, including the Paley-Wiener space 
itself (see the last section).
The third ingredient will be a translation allowing to translate the good parts to a reference
situation (the ``origin'') where we apply the Remez-type inequality. Such translations are
known to be well-behaved for instance in Fock spaces and Paley-Wiener spaces so that 
our construction
easily carries over to those situations. Since in those spaces the result is
already known (\cite{Ko} for the Paley-Wiener space and \cite{JS} for the Fock space), in this 
paper we focus on the case of the 
Bergman spaces.

As a matter of fact, the method proposed here can also be seen as a new way of proving 
sampling results in
the Bergman space (also e.g.\ in Fock or Paley-Wiener spaces). It essentially exploits locally the 
Remez-type inequalities, the proofs of which often 
involve heavy machinery. But once these inequalities established they prove being powerful
tools in sampling problems.

\section{Remez-type inequalities}

In this section we start recalling some results of the paper \cite{AR}.
Let $G$ be a (bounded) domain in $\CC$. Let $0<s<|\overline{G}|$ (Lebesgue measure of $\overline{G}$).
Denoting $\Pol_n$ the space of complex polynomials of degree at most $n\in\NN$, 
we introduce the set
\[
 P_n(\overline{G},s)=\{p\in \Pol_n:|\{z\in \overline{G}:|p(z)|\le 1\}|\ge s\}.
\]
Next, let
\[
 R_n(z,s)=\sup_{p\in P_n(\overline{G},s)}|p(z)|.
\]
This expression gives the biggest possible value at $z$ of a polynomial $p$ of degree at most $n$
and being at most $1$ on a set of measure at least $s$. In particular
Theorem 1 from \cite{AR} claims that for $z\in \partial G$, we have
\bea\label{AR}
 R_n(z,s)\le \left(\frac{c}{s}\right)^n.
\eea
This result corresponds to a generalization to the two-dimensional case of the 
Remez inequality which is usually
given in dimension 1.
In what follows we will essentially consider $G$ to be a disk or a rectangle. 
By the maximum modulus principle, the above constant gives an upper estimate on $G$ for 
an arbitrary polynomial of degree at most $n$ which is bounded by one on a set of
measure at least $s$. Obviously, if this set is small ($s$ close to $0$), i.e. $p$ is controlled by 1 on a 
small set, then the estimate has to get worse.
\\

\begin{remark}\label{rem1}
Let us make another observation. If $c$ is the constant in \eqref{AR} associated with 
the unit disk $G=\DD=D(0,1)$, then a simple argument based on homothecy shows that
the corresponding constant for an arbitrary disk $D(0,r)$ is $cr^2$ (considering $D(0,r)$ as underlying domain, the constant $c$ appearing in \cite[Theorem 1]{AR} satisfies $c>2\times m_2(D(0,r))$). So, in the sequel we will use the estimate
\bea
\label{AR1}
 R_n(z,s)\le \left(\frac{cr^2}{s}\right)^n,
\eea
where $c$ does not depend on $r$.
\end{remark}

Let us recall Lemma 1 from \cite{Ko}.
\begin{lem}[Kovrijkine]\label{LemKo}
Let $\phi$ be analytic in $D(0,5)$ and let $I$ be an interval of length 1 such that $0\in I$ and let
$E\subset I$ be a measurable set of positive measure. If $|\phi(0)|\ge 1$ and 
$M=\max_{|z|\le 4}|\phi(z)|$ then
\[
 \sup_{z\in I}|\phi(z)|\le \left( \frac{C}{|E|}\right)^{\frac{\ln M}{\ln 2}}\sup_{x\in E}|\phi(x)|.
\]
\end{lem}

We will discuss the following counterpart for the planar case:
\begin{lem}\label{Kovr-2D}
Let $0<r<\rho$ be fixed. There exists a constant $\eta>0$ such that the following holds. 
Let $\phi$ be analytic in $D(0,\rho)$, 
let 
$E\subset D(0,r)$ be a planar measurable set of positive measure, and let $z_0\in D(0,r)$. 
If $|\phi(z_0)|\ge 1$ and $M=\max_{|z|\le \rho}|\phi(z)|$ then
\[
 \sup_{z\in D(0,r)}|\phi(z)|\le \left( \frac{cr^2}{|E|}\right)^{\eta \ln M}\sup_{z\in E}|\phi(z)|,
\]
where $c$ does not depend on $r$, and
\[
 \eta \le c''\frac{\rho^4}{(\rho-r)^4}\ln\frac{\rho}{\rho-r}
\]
for an absolute constant $c''$.
\end{lem}

For the case of the Bergman space which is the main object in this paper we can consider
$(1-r)=\kappa (1-\rho)$ for some $\kappa>1$, and, of course, $\rho<1$, so that the above estimate becomes
\bea\label{EstimEta}
 \eta\le c''\frac{1}{(1-r)^4}\ln\frac{1}{1-r},
\eea
where $c''$ is another absolute constant.
\\

Kovrijkine's proof, based on Jensen's inequality and the Remez inequality in the 
one-dimensional case, carries almost verbatim over to the two-dimensional case.
Still, since our geo\-metrical setting is not exactly the same and we need to take into account
the location of $z_0$, and moreover we wish to keep track
of the underlying constants, we reproduce here the adapted version of the proof. Also, in our proof we make
use of the pseudohyperbolic metric which changes a bit the setting and which could
be of interest in other situations. For $a\in \DD$, define the usual disk
automorphism by
\bea\label{autom}
 \varphi_a(z)=\frac{a-z}{1-\overline{a}z},\quad z\in \DD.
\eea
It is well known that $\varphi_a(\varphi_a(z))=z$.

\begin{proof}
We first observe that by a rescaling argument we can assume $\rho=1$.
Let 
\[
 \psi(z)=\phi\circ\varphi_{z_0}(z)=\phi(\frac{z_0-z}{1-\overline{z}_0z}),
\] 
which translates $z_0$ to $0$. Also, letting $r_0=2r/(1+r^2)$ (which corresponds to a 
pseudohyperbolic doubling of $r$), we get $E\subset D_{phb}(z_0,r_0)$.
Define $\tilde{E}=\varphi_{z_0}(E)$ so that $\tilde{E}\subset D(0,r_0)$, and 
$\psi|\tilde{E}=\phi\circ\varphi_{z_0}|\tilde{E}
=\phi|E$.
From now on, we will consider $\psi$ on 
$D(0,r_0)$ which thus contains $\tilde{E}$, and $|\psi(0)|=|\phi(z_0)|\ge 1$.

\includegraphics{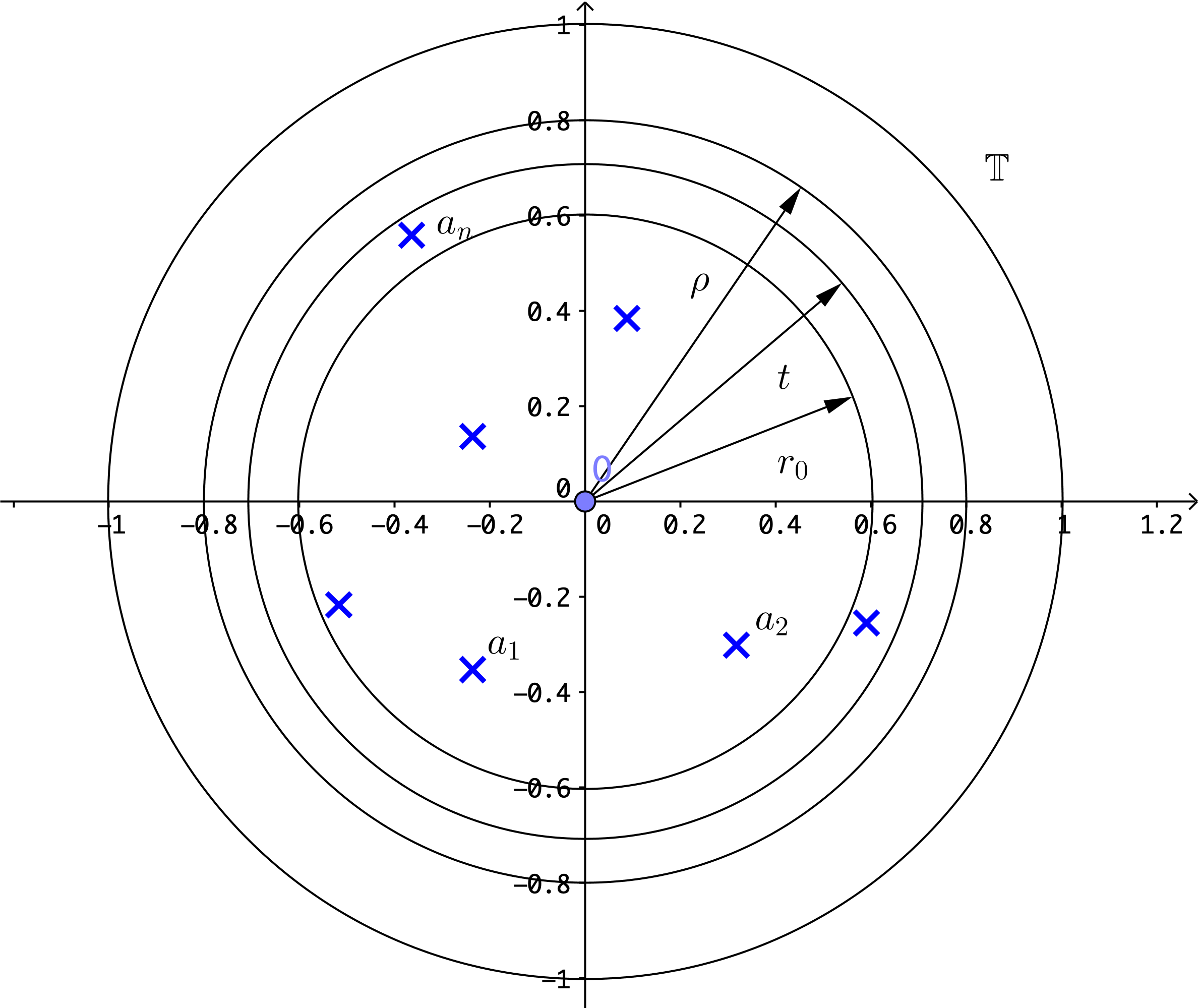}

We now let $t=2r_0/(1+r_0^2)$ another
pseudohyperbolic doubling, now of $r_0$. 
Consider $a_1,\ldots,a_n$, the zeros of $\psi$ in $D(0,t)$, and $B$ 
the Blaschke product associated with these zeros in $D(0,t)$. Then we can divide by 
$B=P/Q$ to obtain a zero-free function $g=\psi/B$ in $D(0,t)$. 
Here  $P=\prod_{k=1}^nt(a_k-z)$ and $Q=\prod_{k=1}^n (t^2-\overline{a}_kz)$ are
 polynomials of degree $n$. Since $|B|=1$ on $\partial D(0,t)$ and $|B|\le 1$ in $D(0,t)$, 
we have $|g(0)|\ge |\psi(0)|=|\phi(z_0)|\ge 1$ and $\max_{|z|\le t}|g(z)|\le M$. Applying Harnack's inequality to the 
positive harmonic function $u=\ln M-\ln |g|$ in $D(0,t)$ we get 
for every $z$ in the compact set $\overline{D(0,r_0)}\subset D(0,t)$ that
\[
 \frac{t-r_0}{t+r_0} u(0)\le  u(z)\le \frac{t+r_0}{t-r_0} u(0)\le \frac{t+r_0}{t-r_0} \ln M.
\] 
Set $\alpha=(t+r_0)/(t-r_0)$ to simplify notation.
We
deduce that for every $z\in \overline{D(0,r_0)}$, $|g(z)|\ge M^{1-\alpha}$. Hence
\[
  \frac{\max_{|z|\le r_0}|g(z)|}{\min_{|z|\le r_0}|g(z)|}\le M/M^{1-\alpha}=M^\alpha.
\]

Let us estimate $Q$:
\[
  \frac{\max_{|z|\le r_0}|Q(z)|}{\min_{|z|\le r_0}|Q(z)|}\le
 \frac{\max_{|z|\le r_0}\prod_{k=1}^n |t^2-\overline{a}_kz|}{\min_{|z|\le r_0}\prod_{k=1}^n |t^2-\overline{a}_kz|}
 \le \left(\frac{t+r_0}{t-r_0}\right)^n
\]
We now use the Andrievskii-Ruscheweyh estimate. Given a polynomial $P$ of degree at most
$n$. Let $m=\sup_{z\in E}|P(z)|$. Set $P_1=P/m$ and $s=|\tilde{E}|$. Then $P_1\in
P_n(\overline{D(0,r_0)},s)$. By \eqref{AR1} we get
\[
 |P_1(z)|\le \left(\frac{cr_0^2}{|\tilde{E}|}\right)^{n},
\] 
and hence
\[
 \sup_{|z|\le r_0}|P(z)|\le \left(\frac{cr_0^2}{|\tilde{E}|}\right)^{n}\sup_{z\in \tilde{E}}|P(z)|.
\]
The exact same estimates as Kovrijkine's lead to 
\beqa
 \sup_{|z|\le r_0}|\psi(z)|&\le&  M^{\alpha}\left(\frac{t+r_0}{t-r_0}\right)^n\times \left(\frac{cr_0^2}{|\tilde{E}|}\right)^{n}
 \times \sup_{z\in \tilde{E}}|\psi(z)|\\
&\le& M^{\frac{t+r_0}{t-r_0}}\left(\frac{t+r_0}{t-r_0}\right)^{n}\times \left(\frac{cr_0^2}{|\tilde{E}|}\right)^{n}
 \times \sup_{z\in \tilde{E}}|\psi(z)|\\
 &\le& e^{\frac{t+r_0}{t-r_0}\ln M+n\ln\frac{t+r_0}{t-r_0}}\left(\frac{cr_0^2}{|\tilde{E}|}\right)^n
 \sup_{z\in \tilde{E}}|\psi(z)|.
\eeqa
Since $\tilde{E}=\varphi_{z_0}(E)$ and $|z_0|<r_0$, we have
$|\tilde{E}|\ge c(1-r_0)^2|E|$, so that
\[
 \sup_{|z|\le r_0}|\psi(z)|\le  
 e^{\frac{t+r_0}{t-r_0}\ln M+n\ln\frac{t+r_0}{t-r_0}-2n\ln (1-r_0)}\left(\frac{cr_0^2}{|{E}|}\right)^n
 \sup_{z\in \tilde{E}}|\psi(z)|.
\]
Assuming $e\le cr_0^2/|{E}|$ (note that this inequality means that the density of ${E}$
is less than a fixed constant, if the density is bigger than some fixed constant, say $1/2$, 
then the estimate is not really of interest).
This yields
\[
 \sup_{|z|\le r_0}|\psi(z)|\le\left(\frac{cr_0^2}{|{E}|}\right)^{\frac{t+r_0}{t-r_0}\ln M+n\ln\frac{t+r_0}{t-r_0}-2n\ln(1-r_0)+n}\sup_{z\in \tilde{E}}|\psi(z)|
\]

Applying Jensen's formula in a similar way as did Kovrijkine we obtain
$n\le \frac{\ln M}{\ln (1/t)}$ (recall that $\rho$ was supposed to be 1). Setting
\[
 \eta=\frac{t+r_0}{t-r_0}+\left(\ln\frac{t+r_0}{t-r_0}-2\ln(1-r_0)+1\right)/\ln(1/t),
\]
and retranslating to $\phi$ we get
\[
 \sup_{|z|\le r}|\phi(z)|\le\sup_{|z|\le r_0}|\psi(z)|
\le\left(\frac{cr_0^2}{|{E}|}\right)^{\eta\ln M}\sup_{z\in \tilde{E}}|\psi(z)|
 \le \left(\frac{cr_0^2}{|{E}|}\right)^{\eta\ln M}\sup_{z\in {E}}|\phi(z)|
\]

Let us discuss the estimate for $\eta$. Note first that $t=2r_0/(1+r_0^2)$, so that 
\[
 \frac{t+r_0}{t-r_0}=\frac{3+r_0^2}{1-r_0^2}\le \frac{4}{1-r_0^2}.
\]
Also $\ln (1/t)\sim (1-r_0)^2/(2r_0)$, and since $r_0=2r/(1+r^2)$ we have
$1-r_0\sim (1-r) ^2/2$.
Hence
\beqa
 \eta&\lesssim&\frac{4}{1-r_0^2}+\left(\ln \frac{4}{1-r_0^2}+\ln \frac{1}{(1-r_0)^2}+1\right)
 \times \frac{2}{(1-r_0)^2}\le \frac{c}{(1-r_0)^2}\ln\frac{1}{1-r_0^2}\\
 &\sim& \frac{c'}{(1-r)^4}\ln\frac{1}{1-r}
\eeqa
where $c$ and $c'$ are absolute constants. 
%

Getting back to arbitrary $\rho>0$, the above becomes
\[
 \eta \le \frac{c''\rho^4}{(\rho-r)^4}\ln\frac{\rho}{\rho-r}
\]
\end{proof}

The corresponding case for $p$-norms is deduced exactly as in Kovrijkine's work.

\begin{coro}\label{CoroKov}
Let $0<r<\rho$ be fixed. There exists a constant $\eta>0$ such that following holds. 
Let $\phi$ be analytic in $D(0,\rho)$,
let 
$E\subset D(0,r)$ be a planar measurable set of positive measure, and let $z_0\in D(0,r)$. 
If $|\phi(z_0)|\ge 1$ and 
$M=\max_{|z|\le \rho}|\phi(z)|$ then
for $p\in [1,+\infty)$ we have
\[
 \|\phi\|_{L^p(D(0,r))}\le \left(\frac{cr^2}{|E|}\right)^{\eta\ln M+\frac{1}{p}}\|\phi\|_{L^p(E)}.
\]
\end{coro}

The estimates on $\eta$ are the same as in the lemma. The constant $c$ does not depend on $r$.


\section{Proof of Theorem \ref{thm1}}\label{proof}

In the proof we will use a change of variable formula. Recall the definition of the
disk-automorphism from \eqref{autom}.
Then we can introduce the change of variable formula
\[
 T_a:A^{p,\alpha}\lra A^{p,\alpha}, \quad (T_af)(z)=(f\circ \varphi)(z)\varphi_a'(z)^{(2+\alpha)/p}. 
\]
Recall that 
\[
 \varphi'_a(z)=-\frac{1-|a|^2}{(1-\overline{a}z)^2}.
\]
Clearly $T_a$ is an isomorphism of the Bergman space $A^{p,\alpha}$. We will also consider
restrictions of $T_a$ from pseudohyperbolic disks to other corresponding pseudohyperbolic
disks.
\\

We now enter into the main part of the proof which gives a simple way of avoiding
the Bernstein inequality. We need to introduce some preliminary notation. For 
$n\in\NN$ and $k=0,1,\ldots,2^n-1$, let $z_{n,k}=(1-2^{-n})e^{i2\pi k/2^n}$ and define
\[
 D_{n,k}^r=D_{phb}(z_{n,k},r),
\]
where $r\in (0,1)$.
It can be checked that for a sufficiently big choice $r_0$ of $r$ these disks cover $\DD$.
We will henceforth set $D_{n,k}=D_{n,k}^{r_0}$.

We need a finite covering property. Denote by $\chi_F$ 
the characteristic function of a measurable set $F$ in $\DD$.

\begin{lem}\label{lem1}
For every $r\in (r_0,1)$, there exists a constant $N$ such that
\[
 \sum_{n,k}\chi_{D^r_{n,k}}\le N.
\]
Moreover there is some universal constant such that 
\[
N\le c_{ov}\frac{1}{(1-r)^2}\ln\frac{1}{1-r},
\]
where $c_{ov}$ is an absolute constant.
\end{lem}

Obviously, the constant $N$ is at least equal to 1. In the estimates below $N$ will 
enter only logarithmically so that the power in $(1-r)$ as well as the logarithmic term
are not very important.

\begin{proof}[Proof of Lemma]
This is certainly a well know fact, but we include a proof for completeness and to give an idea of
the dependence of $N$ on $r$.
 
The result is equivalent to say that a disk $D_{phb}(z,r)$ contains at most $N$ points
$z_{n,k}$. 
Since the problem is rotation invariant, we can assume $z=x\in [0,1)$.

Recall from \cite[p.3]{G} that $D_{phb}(x,r)$ is a euclidean disk with diameter
\[
 [\alpha,\beta]=\left[\frac{x-r}{1-rx},\frac{x+r}{1+rx}\right]
\]

In terms of distance to the boundary, this disk, 
taking into account that $r\in (0,1)$ is fixed, is between
\beqa
 1-\frac{x-r}{1-rx}&=&\frac{1-rx-x+r}{1-rx}
 =\frac{(1+r)(1-x)}{1-rx}\le d_1(1-x),\\
 1-\frac{x+r}{1+rx}&=&\frac{1+rx-x-r}{1+rx}
 =\frac{(1-r)(1-x)}{1+rx}\ge d_2(1-x),
\eeqa
where $d_1$ and $d_2$ are strictly positive constants (e.g. $d_1=2/(1-r)$ and $d_2
=1/d_1=(1-r)/2$).
In particular for $z_{n,k}\in D_{phb}(x,r)$ 
it is necessary that $d_2(1-x)\le 2^{-n}\le d_1(1-x)$ which
may happen at most  $(\ln(d_2(1-x))^{-1}-\ln(d_1(1-x))^{-1})/\ln 2$ times.
With the above choice we can pick
$d_1/d_2\le 4/(1-r)^2$ so that the number of possible $n$ is bounded  
\bea\label{estimN1}
 \ln\frac{d_1}{d_2}/\ln 2\le \ln \frac{4}{(1-r)^2}/\ln 2. 
\eea

We also need to check that the number of $k$ such that $z_{n,k}\in D_{phb}(x,r)$ is uniformly
bounded. Note that $r$ is fixed. 
We will distinguish two cases, $x\le 2r$ et $2r<x<1$ (in case $r\ge 1/2$ the latter case does not occur).

When $x\le 2r$, we know that $D_{phb}(x,r)\subset D(0,3r/(1+2r^2))$. In order that $1-2^{-n}
\in D(0,3r/(1+2r^2))$ it is necessary that $1-2^{-n}\le 3r/(1+2r^2)$ which happens when
$n\lesssim \ln (1-r)^{-1}$ (the constant appearing here does not depend on $r$). In order to get
all $z_{n,k}\in D(0,3r/(1+2r^2))$, it remains to sum $\sum_{k=0}^{\ln (1-r)^{-1}}2^k
\simeq (1-r)^{-1}$. So for $x\le 2r$ we have $N\lesssim (1-r)^{-1}$.

Assume now that $2r<x<1$.
Then
\bea\label{Repart}
 \frac{x-r}{1-rx}\ge \frac{2r-r}{1}=r.
\eea
Observe that in order that $z_{n,k}\in D_{phb}(x,r)$
we need $1-|z_{n,k}|=2^{-n}\in [d_2(1-x),
d_1(1-x)]$. 
By \cite[p.3]{G}, $D_{Phb}(x,r)=D(c,R)$ with
\[
 c=x\frac{1-r^2}{1-r^2x^2},\quad R=r\frac{1-x^2}{1-r^2x^2}.
\]
Hence the argument of a point $w\in D_{phb}(x,r)$ can be estimated in the following way
\beqa
 |\arg w|&\le& \arcsin\frac{R}{c}=\arcsin\left(\frac{r}{x}\frac{1-x^2}{1-r^2}\right)
 \le \arcsin\frac{1-x^2}{2(1-r^2)}\simeq \frac{1-x^2}{2(1-r^2)}\\
 &\le& \frac{1-x}{1-r^2}\le \frac{1}{1-r^2}\frac{1}{d_2}2^{-n}\\
 &\le& \frac{2}{(1-r)^2}2^{-n}.
\eeqa
Hence, in order that $z_{n,k}=(1-2^{-n})e^{2\pi i k/2^n}\in D_{phb}(x,r)$ we need
\bea\label{estimN2}
 2\pi k \le  \frac{2}{(1-r)^2}.
\eea
We conclude that
the number of $k$ is also bounded by a fixed number. 
We are done.

The estimates  \eqref{estimN1} and \eqref{estimN2} yield the required bound on $N$.
\end{proof}

Fix $r_0<s<t<1$.
For $K>1$ the set
\[
 I_f^{K-good}=\{(n,k):\|f\|_{L^{p,\alpha}(D^{t}_{n,k})}\le K \|f\|_{L^{p,\alpha}(D^s_{n,k})}\}
\]
will be called the set of $K$-good disks for $(t,s)$
(in order to keep notation light we will not include
$s$ and $t$ as indices). This set depends on $f$.
\\

The key-result of this paper is the following proposition which allows to obtain
the set of good disks in a very simple way. This might have some independent interest.

\begin{prop}\label{prop1}
Let $r_0\le s<t<1$.
For every constant $c\in (0,1)$, there exists $K$ such that for every $f\in A^{p,\alpha}$ we have
\[
 \sum_{(n,k)\in I_f^{K-good}} \|f\|_{L^{p,\alpha}(D^s_{n,k})}^p\ge c \|f\|_{A^{p,\alpha}}^p.
\]
\end{prop}

It will be clear from the proof that one can pick $K\ge N/(1-c)$ where $N$ corresponds 
to the overlapping constant from Lemma \ref{lem1} for the (pseudohyperbolic) radius $s$.

\begin{proof}[Proof of proposition]
Since $r_0\le s<t<1$, we have $\bigcup_{n,k}D^s_{n,k}=\DD$, and 
by Lemma \ref{lem1} we have a finite overlap property:
\[
 \sum_{n,k}\chi_{D^{t}_{n,k}}\le N.
\]
Hence
\[
 \|f\|_{A^{p,\alpha}}^p\le \sum_{(n,k)}\|f\|^p_{L^{p,\alpha}(D^s_{n,k})}
 \le \sum_{(n,k)}\|f\|^p_{L^{p,\alpha}(D^{t}_{n,k})}
 \le N\|f\|_{A^{p,\alpha}}^p.
\]
Now pick a $c\in (0,1)$.
Then 
\beqa
 N\|f\|_{A^{p,\alpha}}^p
 &\ge& \sum_{(n,k)\notin I_f^{K-good}}\|f\|^p_{L^{p,\alpha}(D_{n,k}^{t})}
 >K^p \sum_{(n,k)\notin I_f^{K-good}}\|f\|^p_{L^{p,\alpha}(D^s_{n,k})}\\
 &=&K^p \sum_{(n,k)}\|f\|^p_{L^{p,\alpha}(D^s_{n,k})} - K^p \sum_{(n,k)\in I_f^{K-good}} 
  \|f\|^p_{L^{p,\alpha}(D^s_{n,k})}\\
 &\ge&K^p\|f\|_{A^{p,\alpha}}^p-K^p\sum_{(n,k)\in I_f^{K-good}} 
  \|f\|^p_{L^{p,\alpha}(D^s_{n,k})}.
\eeqa
For $K^p>N$ we thus get
\[
 \sum_{(n,k)\in I_f^{K-good}} 
  \|f\|^p_{L^{p,\alpha}(D^s_{n,k})}\ge 
 \frac{K^p-N}{K^p}\|f\|_{A^{p,\alpha}}^p
\]
Hence, setting $c=(K^p-N)/K^p\in (0,1)$ we see that $K^p=N/(1-c)$ is suitable.
\end{proof}

We are now in a position to prove the theorem.

\begin{proof}[Proof of Theorem \ref{thm1}]
%

Let $r_1=\max(r,r_0)$, where $r_0$ ensures the covering of $\DD$ by $D_{n,k}$.
It is not hard to see that there exists a universal constant $\eta>0$ 
such that if $E$ is $(\gamma,r)$-dense then it is 
$(\eta\gamma,r_1)$-dense (if $r\ge r_0$ then $\eta=1$, if $r<r_0$ we  can cover
a least portion of any disk $D_{phb}(z,r_0)$ by disjoint disks with radius $r$). 
Since the multiplicative constant $\eta$ does not change 
the estimate claimed in the theorem,  and in order to not overload notation,
in all what follows we will use $\gamma$  instead of $\eta\gamma$.

Now fix $c\in (0,1)$. We will choose $r_1$ for $s$ in Proposition \ref{prop1} and $s<t<1$ to
be fixed later. Let also $K$ be a corresponding choice from Proposition \ref{prop1} (which 
thus depends on $t$ since $N\le c \ln(1/(1-t))/(1-t)^2)$).

Pick $f\in A^{p,\alpha}$. In particular we can assume $\|f\|_{A^{p,\alpha}}=1$.
Let $I_f^{K-good}$ be the set of $K$-good disks for $(r_1,t)$.
Then by Proposition \ref{prop1}
\[
 \|f\|_{A^{p,\alpha}}^p\le \frac{1}{c}\sum_{(n,k)\in I_f^{K-good}}\|f\|_{L^{p,\alpha}(D^{r_1}_{n,k})}^p.
\]
In all what follows we suppose $(n,k)\in I_f^{K-good}$. 
Recall that $z_{n,k}$ is the center of
$D_{n,k}^{r_1}$. 
Clearly $D(0,r_1)=\varphi_{z_{n,k}}(D_{n,k}^{r_1})$ which allows to translate the 
situation on $D_{n,k}^{r_1}$ to a euclidean disk at the origine
(since pseudohyperbolic and euclidean disks
are the same when centered at 0, we will not use the index ``phb" in this situation). Then
the change of variable operator $T_{z_{n,k}}$ restricted to $D_{n,k}^{r_1}$ gives
\[
 \|f\|_{L^{p,\alpha}(D^{r_1}_{n,k})}^p=\int_{D_{n,k}^{r_1}}|f(z)|^pdA_{\alpha}(z)
 =\int_{D(0,r_1)}|(f\circ\varphi_{z_{n,k}})(u) (\varphi'_{z_{n,k}}(u))^{(2+\alpha)/p}|^p dA_{\alpha}(u).
\]
Denote $g=T_{z_{n,k}}f$ and $D_0=D(0,r_1)$. 
Clearly $\int_{D_0}|g|^pdA>0$ so that
we can set 
\[
 h=g\times \left(\frac{\pi r_1^2}{\int_{D_0}|g|^pdA}\right)^{1/p}. 
\]
In particular
 $\int_{D_0}|h|^pdA=\Area(D_0)$, which implies
that there is $z_0\in D_0$ with $|h(z_0)|\ge 1$. 
In order to use Corollary \ref{CoroKov} we have to estimate $h$ on a bigger disk.
Let $\rho=\frac{2r_1}{1+r_1^2}$ be a pseudohyperbolic doubling of the radius $r_1$,
and $t=\frac{2\rho}{1+\rho^2}$ the pseudohyperbolic doubling of $\rho$.
Consider now $h$ 
as a function in the restricted Bergman space $A^{p,\alpha}|D(0,t)$,
we have for $z\in D(0,\rho)$,
\[
 |h(z)|^p\le \frac{C}{(1-\rho)^{(2+\alpha)}}\int_{D(0,t)}
 |h(w)|^p (1-|w|^2)^{\alpha}dA(w),
\]
where $C$ is an absolute constant (the proof is done as in the Bergman space on $\DD$ 
and essentially based on subharmonicity of $|g|^p$). Since $h$ is a constant multiple of
$g$ we can replace in the above $h$ by $g$.
Again, $\rho$ is an iterrated doubling of $r_1$ so that
$1-\rho\simeq (1-r_1)^2$.

Observe that
\[
 \int_{D(0,t)}
 |g(w)|^p dA_{\alpha}(w)
 =\int_{D_{n,k}^t}|f(z)|^p dA_{\alpha}(z),
\]
and
\[
 \int_{D_0}|g|^pdA(z)\ge \int_{D_0}|g|^pdA_{\alpha}(z)=\int_{D_{n,k}^{r_1}}|f|^pdA_{\alpha}(z).
\]
Now, putting all the pieces together, 
\beqa
\lefteqn{
 \sup_{z\in D(0,\rho)}|h(z)| 
=\left( \frac{\pi r_1^2}{\int_{D_0}|g|^pdA(z)}\right)^{1/p}\sup_{z\in D(0,\rho)}|g(z)|}\\
&&\le\left(\frac{\pi}{\int_{D_{n,k}^{r_1}}|f|^pdA_{\alpha}(z)}\right)^{1/p}
\times \frac{C^{1/p}}{(1-\rho)^{(2+\alpha)/p}}
 \left(\int_{D_{n,k}^t}|f(z)|^p dA_{\alpha}(z)\right)^{1/p}.
\eeqa
Since $(n,k)\in I_f^{K-good}$ for $(r_1,t)$ we get
\bea\label{estimM0}
  M=\sup_{z\in D(0,\rho)}|h(z)| \le \frac{DK}{(1-r_1^2)^{2\times(2+\alpha)/p}},
\eea
where $D$ is some universal constant, and 
\[
 K^p=\frac{N}{1-c}\le\frac{c_{ov}}{1-c}
 \frac{1}{(1-t)^2}\ln\frac{1}{1-t}\simeq \frac{1}{(1-r_1)^2}\ln\frac{1}{1-r_1}.
\]
In what follows it will be enough to use the estimate
\bea\label{estimK}
 K\lesssim \frac{1}{(1-r_1)^{3/p}}.
\eea
%

We can now apply
Corollary \ref{CoroKov} to $h$ and $\tilde{E}=\varphi_{z_0}(E\cap D_{n,k})$. Observe that
the density of $\tilde{E}$ in $D_0$ is the same as that of $E$ in $D_{n,k}$ which is thus bounded
below by $\gamma$.
Since we are in a disk of fixed radius we can plug in the weight $(1-|z|^2)^{\alpha}\le 1$,
to the prize of an 
additional constant in the last line:
\beqa
 \int_{z\in D(0,r_1)}|h(z)|^pdA_{\alpha}(z)&\le&
 \int_{z\in D(0,r_1)}|h(z)|^pdA(z)\\
 &\le& \left(\frac{cr_1^2}{|\tilde{E}|}\right)^{p\eta\ln M+1}
 \int_{\tilde{E}} |h(z)|^pdA(z).\\
 &\le&  \left(\frac{c}{\gamma}\right)^{p\eta\ln M+1}\frac{1}{(1-r_2^2)^{\alpha}}
 \int_{\tilde{E}} |h(z)|^pdA_{\alpha}(z).
\eeqa
%
By homogenity we can replace in the above the function $h$ by $g$. 
Thus, changing back variables, we get for every $K$-good disk,
\[
 \int_{D^{r_1}_{n,k}}|f|^pdA_{\alpha}\le
\left(\frac{c}{\gamma}\right)^{p\eta\ln M+1}
 \frac{1}{(1-r_2^2)^{\alpha}}\int_{E\cap D^{r_1}_{n,k}}|f|^pdA_{\alpha},
\]
Since integration over $K$-good disks allows to recover the norm we get the
desired result. 
\end{proof}

A word on the exponent appearing in the estimate.
Recall from \eqref{EstimEta} that
\[
 \eta\le c''\frac{1}{(1-r_1)^4}\ln\frac{1}{1-r_1},
\]

Also from \eqref{estimM0}, \eqref{estimK} and with $\delta=(1-r_1^2)^{(2+\alpha)/p}$,
\[
 \ln M \lesssim \frac{7+2\alpha}{p}\ln\frac{1}{1-r_1^2}
\]
(the constants involved here do not depend on $r_1$)
so that with $7+2\alpha\lesssim 1+\alpha$, we get
\[
 \eta\ln M\lesssim \frac{1+\alpha}{p}\times \frac{1}{(1-r_1)^4} \ln\frac{1}{1-r_1^2}.
\]

\section{Comments on other spaces}

\subsection{Fock spaces}

In Fock spaces (see \cite{Zhu}), 
the situation follows as a special case of the recent work by \cite{JS}. Still,
since our method is quite elementary and universal, we feel interesting to present this 
application of the above techniques.
Let us give the necessary indications on how to obtain the estimates in the
Fock space defined by
\[
 \mathcal{F}^{p,\alpha}=\{f\in \Hol(\CC): \|f\|_{\mathcal{F}^{p,\alpha}}^p
 =\int_{\CC}|f(z)|^pe^{-p\alpha|z|^2/2}dA(z)<+\infty\}.
\]
Les us also use $dA_{\alpha,p}(z)=e^{-p\alpha |z|^2/2}dA(z)$.
In the Fock space the computations are easier since we do
not need to work with pseudohyperbolic distance but with euclidean distance. 
In this situation pick $z_{n,k}=n+ik$, $n,k\in\ZZ$. Replace the pseudohyperbolic
disks by euclidean ones $D_{n,k}^r=D(z_{n,k},r)$ where now $r>\sqrt{2}$ will ensure the 
covering property. For every fixed $r>\sqrt{2}$ it is clear that the finite overlap property holds
(we are in euclidean geometry). One easily sees that $N\simeq r^2$. 
Also there is an isometric translation (change of variable)
operator $T_af(z)=e^{\alpha\overline{a}z-\alpha |a|^2/2}f(z-a)$ 
which will allow to translate the situation
from an arbitrary disk $D_{n,k}^r$ to $D(0,r)$.

As in the Bergman space, we introduce good disks. Fix $\sqrt{2}<s<t$.
For $K>1$ and a function $f\in \mathcal{F}^{p,\alpha}$ the set
\[
 I_f^{K-good}=\{(n,k):\|f\|_{L^{p,\alpha}(D^{t}_{n,k})}\le K \|f\|_{L^{p,\alpha}(D^s_{n,k})}\}
\]
will be called the set of $K$-good disks for $(t,s)$, where we integrate with respect to
the measure $dA_{\alpha,p}$. In the argument below we pick $s=r$ and $t=4r$.


Now, given $f$ with $\|f\|_{\mathcal{F}^{p,\alpha}}=1$, let,
as in the Bergman space, $g=T_{-z_{n,k}}f$ for $(n,k)\in I_f^{K-good}$, and set
\[
 h=c_0g, \quad c_0=\left(\frac{\pi r^2}{\int_{D(0,r)}|g|^pdA(z)}\right)^{1/p}.
\]
Again there is $z_0\in D(0,r)$ with $|h(z_0)|\ge 1$. 

Set $\rho=2r$. We have to estimate the maximum modulus of $h$ on $D(0,\rho)$ in terms
of a local integral of $h$. To that purpose, we can assume $h\in  A^p(D(0,4r))$ 
which justifies the first of the estimates below.
Hence
\beqa
 \max_{z\in D(0,\rho)}|h(z)|^p&\le& \frac{C}{r^2} \int_{D(0,4r)}|h|^pdA(z)\\
 &=& \frac{C}{r^2}\times\frac{\pi r^2}{\int_{D(0,r)}|g|^pdA(z)} \int_{D(0,4r)}|g|^pdA(z)\\
 &\le& \frac{Ce^{8\alpha p r^2}}{r^2} \frac{\pi r^2}{\int_{D(0,r)}|g|^pdA_{\alpha,p}(z)} 
 \int_{D(0,4r)}|g|^pdA_{\alpha,p}(z)
\eeqa
(in the last estimate, in order to switch from $dA$ to $dA_{\alpha,p}$ we had to introduce an
additional factor $e^{\alpha p (4r)^2/2}$).
 Since $(n,k)$ is $K$-good for $(r,4r)$, we see that
the last expression in the above inequalities is bounded by a constant (depending on $r$)
times $K^p$: $M^p =\max_{z\in D(0,\rho)}|h(z)|^p\le c_rK^p$, where
$c_r=C\pi e^{8\alpha p r^2}$. Exactly as in Proposition \ref{prop1} we see that
$K^p\simeq N\lesssim r^2$, so that 
\bea\label{estimM}
 M\lesssim r^{2/p}e^{8\alpha r^2}.
\eea

Hence, setting $\tilde{E}=(E\cap D^r_{n,k})-z_{n,k}$ 
we get using Corollary \ref{CoroKov} applied to $h$:
\[
\int_{D(0,r)}|h(z)|^pdA(z) \\
 \le\left(\frac{cr^2}{|\tilde{E}|}\right)^{p\eta\ln (M/\delta)+1}
 \int_{\tilde{E}}|h(z)|^pdA(z)\\
\]
The factor $r^2$ appearing inside the brackets is a rescaling factor (see Remark \ref{rem1}).
Again, by homogenity we can replace in the above inequality $h$ by $g$.
Note also that $\pi r^2/|\tilde{E}|$ is controlled by $1/\gamma$.
This yields

\beqa
 \int_{D^{r}_{n,k}}|f|^p
 dA_{\alpha, p}(z)&=&\int_{D(0,r)}|g(z)|^p 
 dA_{\alpha, p}(z)
\le \int_{D(0,r)}|g(z)|^pdA\\ 
&\le&\left(\frac{cr^2}{|\tilde{E}|}\right)^{p\eta\ln M+1}
 \int_{\tilde{E}}|g(z)|^pdA(z)\\
&\le&  e^{2\alpha pr^2} 
\left(\frac{c_1}{\gamma}\right)^{p\eta\ln M+1}\int_{\tilde{E}}|g(z)|^pe^{-p\alpha|z|^2/2}dA(z)\\
&=&e^{2\alpha pr^2} \left(\frac{c_1}{\gamma}\right)^{p\eta\ln M+1}
 \int_{E\cap D^r_{n,k}}|f(z)|^p dA_{\alpha,p}, 
\eeqa
where $c_1$ is an absolute constant.

Summing over all $K$-good $(n,k)$ we obtain the required result
\[
  \|f\|_{\mathcal{F}^{p,\alpha}}\lesssim e^{2\alpha r^2}
 \left(\frac{c_1}{\gamma}\right)^{\eta\ln M+1/p}\|f\|_{L^{p}(E,dA_{\alpha,p})}
\]
where 
\[
 \ln M\lesssim 8\alpha r^2\ln r+\frac{2}{p}\ln r,
\] 
and, in view of \eqref{EstimEta},
\[
\eta =c''\times 2^4\ln 2.
\]

\begin{remark}
In the above reasoning we had fixed the centers to be $z_{n,k}=n+ik$ which makes them independent 
of $r$. Still, one could replace these by $\tilde{z}_{n,k}=rn+irk$ (and consider the disks with these $\tilde{z}_{n,k}$ and radius $(\sqrt{2}+\epsilon)r)$), which will give a better control on the overlapping constant $N$ (actually $N\le 4$ in this case). Still replacing the radii $2r$ and $4r$ by slightly bigger or smaller ones will easily kill the term $r^2$ in the estimate of $M$ \eqref{estimM}
\end{remark}

\subsection{Paley-Wiener spaces}
The above arguments apply also to the Paley-Wiener spaces, in particular to find the good
intervals. We do not claim that our proof is better than Kovrijkine's nor that we get better
constants (also, Kovrijkine's key lemma is still used). We would just like to point out that also in this case Bernstein's inequality is not required 
to run the proof.
Let us recall the definition of Paley-Wiener spaces:
\[
 PW^p_b=\{f\in \Hol(\CC):f\text{ is of exponential type at most }b,\
 \|f\|_{PW_b^p}^p=\frac{1}{\pi}\int_{\RR}|f(x)|^pdx<\infty\}
\] 
where $b>0$ is the bandwith (often chosen to be $\pi$).

We will now use Proposition \ref{prop1} to find the good intervals. 
Let us give some indications for this case.
We first observe that an equivalent norm in the Paley-Wiener space can be given by integrating
over a strip: $S^h=\{z=x+iy\in\CC:x\in\RR,|y|<h\}$ for some fixed $h>0$:
\[
 \|f\|_{PW^p_b}^p\simeq \int_{S^h}|f(z)|^pdA(z).
\]
Here the constants depend on $b$, $p$ and $h$ (in what follows we will choose $h=10$).
This can be seen from the Plancherel-Polya estimate (see \cite[p.96]{S}), or checking that $\chi_SdA$, where
$\chi_S$ is the characteristic function of the strip $S$, is a Carleson and reverse Carleson 
measure for $PW_b^p$ (see \cite{HJK} and \cite{FHR}). The Plancherel-Polya estimate is
certainly a more elementary tool here. Clearly for every $h_1$ and $h_2$ we have
$\|f\|_{L^p(S^{h_1})}\simeq \|f\|_{L^p(S^{h_2})}$ for $f\in PW_b^p$, and the constants in play
only depend on $h_1$ and $h_2$.

Relative $(\gamma,r)$-density of a measurable set $E\subset \RR$ (Lebesgue measure on $\RR$)  
here means that 
\[
 \frac{|I\cap E|}{|I|}\ge\gamma
\]
for every interval $I$ of length $r$. We can cover the substrip of $S^r\subset S^{10r}$ 
by rectangles $R_{n,r}$
with length $r$ and height $2r$. 
For this, pick $R_{n,r}=[z_n-r/2,z_n+r/2]\times 
[-r,r] $ and $z_n=rn$, $n\in\ZZ$ (this will cover $S^r$ up to a set with zero planar 
Lebesgue measure).
The following adaption of Kovrijkine's statement will be useful. It is proved exactly as 
\cite[Lemma 1]{Ko}.
\begin{lem}[Kovrijkine-bis]\label{LemKoBis}
Fix $r>0$.
Let $\phi$ be analytic in $D(0,10r)$ and let $I=[-\D\frac{r}{2},\D\frac{r}{2}]$. Let
$z_0\in D(0,r)$, and
$E\subset I$ a measurable set of positive measure. If $|\phi(z_0)|\ge 1$ and 
$M=\max_{|z|\le 9r}|\phi(z)|$ then
\[
 \sup_{z\in I}|\phi(z)|\le \left( \frac{C}{|E|}\right)^{\frac{\ln M}{\ln 2}}\sup_{x\in E}|\phi(x)|.
\]
\end{lem}

Again, there is also an $L^p$-version of this result:
\begin{coro}
Under the conditions of the lemma, and let $1\le p<+\infty$, we have
\[
 \|\phi|_{L^p(I)} \le \left( \frac{C}{|E|}\right)^{\frac{\ln M}{\ln 2}+\frac{1}{p}}\|\phi\|_{L^p(E)}.
\]
\end{coro}

Now, as in Proposition \ref{prop1}
we find a set of $K$-good rectangles satisfying now $\int_{R_{n,10r}}|f|^pdA\le K
\int_{R_{n,r}}|f|^pdA$ (the arguments
are exactly the same, but one has to take care of the fact that we use two different norms
associated with
$S^{h_1}$ and $S^{h_2}$). Note that we could have chosen disks, but rectangles, centered
on a given discrete set (here $r\ZZ$), are more adapted for the covering of the strip. 
On these rectangles we use the same kind of estimates
as before to control the maximum on $R_{n,r}$ in a uniform manner depending only
on the local norm.
More precisely, given $f\in PW_b^p$ and a $K$-good rectangle $R_{n,r}$, set
\[
 h=\left(\frac{\pi r^2}{\int_{D({z_n},r)}|f|^pdA}\right)^{1/p}f.
\]
Then there is $z_0\in D(z_n,r)$ with $|h(z_0)|\ge 1$, and as above we can estimate
(considering $f$ as a function on the Bergman space $D(z_n,10r)$)
\[
 \max_{D(z_n,9r)} |h(z)|^p\le c\int_{D(z_n,10r)}|h|^pdA
 =c\pi r^2 \int_{D(z_n,10r)}|f|^pdA/\int_{D(z_n,r)}|f|^pdA\le cK\pi r^2.
\]

From here on, the rest follows as in Kovrijkine's argument using Lemma \ref{LemKoBis}
and its version for $L^p$.

\subsection{Dirichlet spaces}

It is not completely clear how to define dominating sets for Dirichlet spaces. Recall that (weighted) Dirichlet spaces, or more general Besov spaces can be defined by 
\[
 \mathcal{B}^p_{\alpha}=\{f\in \Hol(\DD):\|f\|^p_{\mathcal{B}^p_{\alpha}}=|f(0)|^p+\int_{\DD}|f'(z)|^p(1-|z|^2)^{\alpha}dm<\infty\}.
\]
Defining a positive, finite measure $\mu$ on $\DD$ as reverse Carleson measure by asking 
$\int_{\DD}|f|^pd\mu\ge C\|f\|_{\mathcal{B}^p_{\alpha}}^p$, thus generalizing \eqref{dom}
to arbitrary positive measures, leads to an 
empty result in general. Indeed, at least for $p=2$ and $\alpha=0$ it was mentioned after 
\cite[Theorem 8.3]{FHR} that such measures do simply not exist. If, instead, one 
defines dominating sets as measurable sets $E\subset \DD$ containing $0$ such that
\[
 \|f\|^p_{\mathcal{B}^p_{\alpha}}=|f(0)|^p+\int_{\DD}|f'(z)|^p(1-|z|^2)^{\alpha}dm
 \ge c |f(0)|^p+\int_{E}|f'(z)|^p(1-|z|^2)^{\alpha}dm
\]
then the key observation is that $f\in 
\mathcal{B}^p_{\alpha}$ if and only if $f'\in A^{p,\alpha}$ and use the results found in the 
Bergman space to get the same sampling constant estimates
(which will thus give the constant when $0\in E$; otherwise one could add $0$ or an arbitrarily
small neighborhood of 0 which does not change the density).


\begin{thebibliography}{99}

\bibitem{AR}
V.V. Andrievskii \& S. Ruscheweyh,
{\it Remez-type inequalities in the complex plane},
Constr. Approx. \textbf{25} (2007), no. 2, 221-237.

\bibitem{BE}
P. Borwein \& T. Erdelyi, 
{\it Polynomials and polynomial inequalities.} 
Graduate Texts in Mathematics, 161. Springer-Verlag, New York, 1995. x+480 pp.

\bibitem{E}
T. Erd\'elyi, {\it Remez-type inequalities and their applications},
J. Comput. Appl. Math. \textbf{47} (1993), no. 2, 167-209. 

\bibitem{FHR}
E. Fricain, A. Hartmann, W.T. Ross, {\it A survey on reverse Carleson measures},
Harmonic Analysis, operator theory, function theory, and applications, Jun 2015, Bordeaux, France. pp.91-123

\bibitem{G}
J. Garnett, {\it Bounded analytic functions}. Revised first edition. Graduate Texts in Mathematics, 236. Springer, New York, 2007. xiv+459 pp.

\bibitem{HJK}
A. Hartmann, Ph. Jaming, \& K. Kellay,
{\it Quantitative estimates of sampling constants}, accepted for publication in
Amer. J. Math, arXiv: 1707.07880.

\bibitem{HKZ}
H. Hedenmalm, B. Korenblum,\& K. Zhu, {\it Theory of Bergman spaces}. 
Graduate Texts in Mathematics, 199. Springer-Verlag, New York, 2000. x+286 pp.

\bibitem{JS}
Ph. Jaming \& M. Speckbacher,
{\it Planar sampling sets of the short-time Fourier transform}, preprint,
arXiv: 1906.02964.

\bibitem{Ko}
O. Kovrijkine
{\it Some results related to the Logvinenko-Sereda theorem}. 
Proc. Amer. Math. Soc. \textbf{129} (2001), no. 10, 3037-3047.


\bibitem{LS}
V. N. Logvinenko \& Yu. F. Sereda,
{\it Equivalent norms in spaces of entire functions of exponential type}. 
Teor. Funktsii, Funktsional. Anal. i Prilozhen \textbf{19} (1973), 234-246.

\bibitem{Lu1}
D. H. Luecking, 
{Inequalities on Bergman spaces}. Illinois J. Math. 25 (1981), no. 1, 1-11.

\bibitem{Lu2}
D. H. Luecking,
{\it Forward and reverse Carleson inequalities for functions in Bergman spaces and their 
derivatives}. Amer. J. Math. 107 (1985), no. 1, 85-111.

\bibitem{Lu3}
D. H. Luecking, 
{\it Dominating measures for spaces of analytic functions}. Illinois J. Math. 32 (1988), no. 1, 23-39.


\bibitem{OC}
J. Ortega-Cerd\`a,
{\it Sampling measures}. Publ. Mat. \textbf{42} (1998), no. 2, 559-566.

\bibitem{Pa1}
B. P. Panejah,
{\it On some problems in harmonic analysis}. 
Dokl. Akad. Nauk SSSR, 142 (1962), 1026-1029.

\bibitem{Pa2}
B. P. Panejah,
{\it Some inequalities for functions of exponential type and a priori estimates for general differential operators}. Russian Math. Surveys \textbf{21} (1966), 75-114.

\bibitem{Po}
A. Poltoratski,
{\it Toeplitz methods in completeness and spectral problems}, Proc. Int. Cong. of Math. 2018,
Rio de Janeiro, Vol.2, 1739-1774.

\bibitem{Ru}
Walter Rudin, {Real and complex analysis}. Third edition.
McGraw-Hill Book Co., New York, 1987. xiv+416 pp.

\bibitem{S}
K. Seip,
{\it Interpolation and sampling in spaces of analytic functions}. University Lecture Series, 33. American Mathematical Society, Providence, RI, 2004. xii+139 pp

\bibitem{Zhu}
K. Zhu, {\it Analysis on Fock spaces}, Graduate Texts in Mathematics, 263. Springer, 
New York, 2012. x+344 pp.
\end{thebibliography}
\end{document}